\documentclass[12pt]{amsart}
\usepackage{amsfonts,amssymb,amscd,amstext,mathrsfs}

\input xy
\xyoption{all}

\newtheorem{theorem}{Theorem}[section]
\newtheorem{lemma}[theorem]{Lemma}
\newtheorem{corollary}[theorem]{Corollary}
\newtheorem{proposition}[theorem]{Proposition}

\theoremstyle{definition}
\newtheorem{example}[theorem]{Example}

\newcommand{\C}{\mathbb C}
\newcommand{\N}{\mathbb N}
\newcommand{\Q}{\mathbb Q}
\newcommand{\Z}{\mathbb Z}
\renewcommand{\O}{\mathscr O}
\renewcommand{\ker}{\operatorname{Ker}}
\newcommand{\im}{\operatorname{Im}}
\newcommand{\supp}{\operatorname{supp}}
\newcommand{\id}{\operatorname{id}}

\begin{document}

\title{Extending holomorphic maps from Stein manifolds \\ into affine toric varieties}

\author{Richard L\"ark\"ang}
\address{Richard L\"ark\"ang, School of Mathematical Sciences, University of Adelaide, Adelaide SA 5005, Australia}
\curraddr{Department of Mathematics, University of Wuppertal, Gau\ss str.~20, 42119 Wuppertal, Germany, and Department of Mathematics, Chalmers University of Technology and the University of Gothenburg, 412 96 Gothenburg, Sweden}
\email{larkang@chalmers.se}

\author{Finnur L\'arusson}
\address{Finnur L\'arusson, School of Mathematical Sciences, University of Adelaide, Adelaide SA 5005, Australia}
\email{finnur.larusson@adelaide.edu.au}

\thanks{The authors were supported by Australian Research Council grant DP120104110.}


\date{28 October 2014.  Most recent minor changes 24 January 2016}


\begin{abstract}  
A complex manifold $Y$ is said to have the interpolation property if a holomorphic map to $Y$ from a subvariety $S$ of a reduced Stein space $X$ has a holomorphic extension to $X$ if it has a continuous extension.  Taking $S$ to be a contractible submanifold of $X=\C^n$ gives an ostensibly much weaker property called the convex interpolation property.  By a deep theorem of Forstneri\v c, the two properties are equivalent.  They (and about a dozen other nontrivially equivalent properties) define the class of Oka manifolds.  

This paper is the first attempt to develop Oka theory for singular targets.  The targets that we study are affine toric varieties, not necessarily normal.  We prove that every affine toric variety satisfies a weakening of the interpolation property that is much stronger than the convex interpolation property, but the full interpolation property fails for most affine toric varieties, even for a source as simple as the product of two annuli embedded in $\C^4$.
\end{abstract}

\maketitle
\tableofcontents

\section{Introduction} 
\label{sec:intro}

\noindent 
Modern Oka theory has evolved from Gromov's seminal work on the Oka principle \cite{Gromov1989}.\footnote{The monograph \cite{Forstneric2011} is a comprehensive reference on Oka theory.  See also the surveys \cite{Forstneric2013} and \cite{FL2011}.}  From one point of view, Oka theory is the study of interpolation and approximation problems for holomorphic maps from Stein spaces into suitable complex manifolds.  The goal, for suitable targets, is to show that such a problem can be solved as soon as there is no topological obstruction to its solution.  A complex manifold $Y$ is said to have the \textit{interpolation property} if a holomorphic map to $Y$ from a subvariety $S$ of a reduced Stein space $X$ has a holomorphic extension to $X$ if it has a continuous extension.  Taking $S$ to be a contractible submanifold of $X=\C^n$ gives an ostensibly much weaker property called the \textit{convex interpolation property}.  It is a deep theorem of Forstneri\v c that the two properties are equivalent \cite{Forstneric2009}.  In fact, by Forstneri\v c's work, a dozen or more properties of complex manifolds having to do with interpolation or approximation or both are mutually equivalent.  They define the class of Oka manifolds.  The prototypical examples of Oka manifolds are complex Lie groups and their homogeneous spaces.  Among other known examples are all smooth toric varieties (\cite{Larusson2011}, \cite[Theorem 2.17]{Forstneric2013}).  In particular, the smooth locus of a toric variety is Oka.

This paper is the first attempt to develop Oka theory for singular targets.  We focus on interpolation, but our interpolation results imply approximation results.  The simple example of the cusp $Y=\{(z,w)\in\C^2:z^2=w^3\}$ shows that Oka theory for singular targets will be different from the theory for smooth targets.  The identity map $Y\to Y$ has a continuous extension to $\C^2$ because $Y$ is contractible (so $\C^2$ retracts continuously onto $Y$), but it does not extend holomorphically to $\C^2$ because being singular, $Y$ is not a holomorphic or even a smooth retract of $\C^2$.  Using the normalisation map $\C\to Y$, $\zeta\mapsto(\zeta^3,\zeta^2)$, it is easily seen that $Y$ satisfies the interpolation property for normal sources $S$.  However, Theorem \ref{t:first-dichotomy}(b) below shows that there is more to Oka theory for singular targets than restricting the sources to being smooth.

The targets that we will study are affine toric varieties, always assumed irreducible but \textit{not necessarily normal} (\lq\lq nnn\rq\rq).  Our main results, stated precisely just below, may be roughly summarised as follows.
\begin{itemize}
\item  Every nnn affine toric variety satisfies a weakening of the interpolation property that is much stronger than the convex interpolation property.
\item  The full interpolation property fails for most nnn affine toric varieties, even for a source as simple as the product of two annuli embedded in $\C^4$.
\end{itemize}

Our main results are the following four theorems.

\begin{theorem}  \label{t:main-positive}
Let $Y$ be a nnn affine toric variety.  Let $S$ be a factorial subvariety of a reduced Stein space $X$ such that $H^p(X,\Z)\to H^p(S,\Z)$ is surjective for $p=0, 1, 2$.  Then every holomorphic map $S\to Y$ extends to a holomorphic map $X\to Y$.
\end{theorem}

As a very particular consequence, every nnn affine toric variety satisfies the convex interpolation property.  The key to the proof of the theorem is a new notion of a \textit{twisted factorisation} of a nondegenerate holomorphic map into $Y$ (see Theorem \ref{t:twisted-factorisation}).  We call a holomorphic map $f:S\to Y\subset\C^n$ \textit{nondegenerate} if the image by $f$ of each irreducible component of $S$ intersects the torus in $Y$.  Equivalently, no component of $f$ is identically zero on any component of $S$.  (As described in more detail at the beginning of Section \ref{sec:proofs}, we take a nnn affine toric variety $Y$ in $\C^n$ to be embedded in $\C^n$ as the zero set of a prime lattice ideal.  Then the torus in $Y$ is $Y\cap(\C^*)^n$.)

Recall that $S$ is factorial if the stalk $\mathscr O_x$ of the structure sheaf of $S$ is a factorial ring for every $x\in S$.  Then $S$ is normal and hence locally irreducible, so its connected components and irreducible components are the same.  Factoriality is a strong property, not much weaker than smoothness, but it is a natural assumption in Theorem \ref{t:main-positive}.

\begin{theorem}  \label{t:first-dichotomy}
Let $Y$ be a nnn affine toric variety of dimension $d$ in $\C^n$ with $0\in Y$.
\begin{enumerate}
\item[(a)]  Suppose that the normalisation of $Y$ is $\C^d$.  If $S$ is a normal subvariety of a reduced Stein space $X$, then every nondegenerate holomorphic map $S\to Y$ extends to a holomorphic map $X\to Y$.
\item[(b)]  Suppose that the normalisation of $Y$ is not $\C^d$.  There is a smooth surface $S$ in $\C^4$, biholomorphic to the product of two annuli, and a nondegenerate holomorphic map $S\to Y$ that does not extend to a holomorphic map $\C^4\to Y$.
\end{enumerate}
\end{theorem}

Here, $0$ denotes the origin in $\C^n$.  A variety $Y$ as in the theorem is contractible (not only topologically, but even holomorphically), so the extension problem has no topological obstruction.  When $X$ and $S$ are smooth, it is easily seen using a smooth retraction of a neighbourhood of $S$ onto $S$, that the extension problem has no smooth obstruction either.

The following dichotomy refines the first case of the previous one.  Note that case (a) refers to arbitrary, possibly degenerate holomorphic maps.  After the proof of Theorem \ref{t:second-dichotomy}, we give an example to show that for (b) to hold in general, the source $\C\times\{0\}\cup\{(0,1)\}$ needs to be disconnected.

\begin{theorem}  \label{t:second-dichotomy}
Let $Y$ be a nnn affine toric variety of dimension $d$ in $\C^n$ with $0\in Y$.  Suppose that the normalisation of $Y$ is $\C^d$.
\begin{enumerate}
\item[(a)]  Suppose that $Y$ is locally irreducible.  If $S$ is a seminormal subvariety of a reduced Stein space $X$, then every holomorphic map $S\to Y$ extends to a holomorphic map $X\to Y$.
\item[(b)]  Suppose that $Y$ is not locally irreducible.  There is a degenerate holomorphic map $\C\times\{0\}\cup\{(0,1)\}\to Y$ that does not extend to a holomorphic map $\C^2\to Y$.
\end{enumerate}
\end{theorem}

The following examples illustrate the three kinds of varieties in Theorems \ref{t:first-dichotomy} and~\ref{t:second-dichotomy}.
\begin{itemize}
\item  The map $\C^2\to\C^4$, $(s,t)\mapsto (s, t^2, t^3, st)$, is the normalisation of its image $Y$.  The map induces a homeomorphism $\C^2\to Y$, so $Y$ is locally irreducible.
\item  Whitney's umbrella in $\C^3$, defined by the equation $x^2y=z^2$, has normalisation $\C^2\to Y$, $(s,t)\mapsto (s, t^2, st)$, and is not locally irreducible.
\item  The cone in $\C^3$ defined by the equation $xy=z^2$ is normal but of course not isomorphic to $\C^2$.  (See Example \ref{x:cone}.)
\end{itemize}

Finally, in the normal case, we have the following result.

\begin{theorem}  \label{t:main-normal}
Let $Y$ be a singular nondegenerate normal affine toric variety.
\begin{enumerate}
\item[(a)]  Let $S$ be a connected factorial subvariety of a connected reduced Stein space $X$ such that $H^2(X,\Z)\to H^2(S,\Z)$ is surjective.  Then every holomorphic map $S\to Y$ extends to a holomorphic map $X\to Y$.
\item[(b)]  There is a smooth surface $S$ in $\C^4$, biholomorphic to the product of two annuli, and a nondegenerate holomorphic map $S\to Y$ that does not extend to a holomorphic map $\C^4\to Y$.
\end{enumerate}
\end{theorem}

A normal affine toric variety may be factored as $Y\times(\C^*)^k$, where the normal affine toric variety $Y$ has no torus factors, and is thus said to be nondegenerate, and $Y$ is naturally embedded in some $\C^n$ so as to contain the origin.  Then $Y$ is contractible, and $Y$ is singular unless $Y$ is isomorphic to some $\C^m$.  The extension problem into $\C^*$ is of course fully understood: the hypothesis on $H^1$ in Theorem \ref{t:main-positive} is only there to take care of torus factors.  Thus Theorem \ref{t:main-normal} follows from Theorems \ref{t:main-positive} and \ref{t:first-dichotomy}.

We prove the main theorems in Section \ref{sec:proofs}.  We also give quick proofs of the following further results.

It is well known in Oka theory that interpolation yields approximation.

\begin{proposition}  \label{p:approximation}
Let $\Omega$ be a Runge domain of finite embedding dimension in a factorial Stein space $X$, such that $H^p(\Omega,\Z)=0$ for $p=1, 2$.  Let $Y$ be a nnn affine toric variety, and let $f:\Omega\to Y$ be a holomorphic map.  Then $f$ can be approximated uniformly on compact subsets of $\Omega$ by holomorphic maps $X\to Y$.
\end{proposition}

Interpolation on a discrete subset is always possible.

\begin{proposition}  \label{p:discrete}
Let $S$ be a discrete subset of a reduced Stein space $X$.  Let $Y$ be a nnn affine toric variety.  Then every map $S\to Y$ extends to a holomorphic map $X\to Y$.
\end{proposition}

The following result is known \cite[Propostion 4.2.6]{CLS2011}, but we offer a new proof.

\begin{proposition}  \label{p:factorial-implies-smooth}
A factorial affine toric variety is smooth.
\end{proposition}

Finally, meromorphic extensions are readily obtained.

\begin{proposition}  \label{p:meromorphic}
Let $S$ be a factorial subvariety of a reduced Stein space $X$.  Let $Y\subset \C^n$ be a nnn affine toric variety with $0\in Y$.  Let $f:S\to Y$ be a nondegenerate holomorphic map.  Then $f$ extends to a meromorphic map $X\to Y$.
\end{proposition}

\section{Proofs}
\label{sec:proofs}

\noindent
A nnn affine toric variety $Y$ in $\C^n$ may be defined in several equivalent ways (see \cite[Chapter 1]{CLS2011} and \cite{Mustata2004}).  We prefer to take $Y$ to be defined by a prime lattice ideal, that is, a prime ideal in $\C[x_1,\ldots,x_n]$ of the form
\[ I=\langle x^{\ell^+}-x^{\ell^-} : \ell\in \Lambda\rangle = \langle x^\alpha-x^\beta:\alpha,\beta\in\N^n, \alpha-\beta\in \Lambda \rangle, \]
where $\Lambda$ is a lattice in $\Z^n$, that is, an additive subgroup of $\Z^n$ (here, $0\in\N$).  Each $\ell\in \Lambda$ is uniquely decomposed as $\ell^+-\ell^-$, where $\ell^+,\ell^-\in\N^n$ are minimal.  The ideal $I$ is prime if and only if the quotient group $\Z^n/\Lambda$ is torsion-free or, equivalently, free \cite[Theorem 7.4]{MS2005}.  When we speak of a nnn affine toric variety $Y$ in $\C^n$, we always take it to be embedded in $\C^n$ as the zero set of a prime lattice ideal.  Then the torus in $Y$ is $Y\cap(\C^*)^n$.  Let $d$ denote the dimension of $Y$. 

The normalisation of $Y$ is most conveniently described in terms of the affine semigroup corresponding to $Y$.  An affine semigroup $\Sigma$ is a finitely generated, torsion-free, cancellative, abelian monoid.  Equivalently, $\Sigma$ is the image of a monoid morphism $\mu:\N^n\to\Z^m$, and we might as well assume that $\Sigma$ generates $\Z^m$ as a group.  Every nnn affine toric variety is the spectrum of the semigroup algebra $\C[\Sigma]$ of some affine semigroup $\Sigma$.  The associated lattice is $\Lambda=\{u-v\in\Z^n:u,v\in\N^n, \mu(u)=\mu(v)\}$.  The assignment of $\operatorname{spec}\C[\Sigma]$ to $\Sigma$ defines an anti-equivalence from the category of affine semigroups to the category of nnn affine toric varieties.  The normalisation of the nnn affine toric variety defined by $\Sigma$ is defined by the saturation of $\Sigma$ in $\Z^m$, that is, by the semigroup of all $v\in\Z^m$ with $nv\in\Sigma$ for some $n\in\N$, $n\geq 1$, which is in fact an affine semigroup.

A saturated affine semigroup is of the form $\sigma^\vee \cap \Z^m$, where $\sigma$ is a rational polyhedral cone.  The saturation of a general affine semigroup $\Sigma$ is $\Sigma_{\rm sat} = \tau \cap \Z^m$, where $\tau$ is the rational polyhedral cone spanned by generators of $\Sigma$. The dual of a rational polyhedral cone is a rational polyhedral cone, and for a rational polyhedral cone $\tau$, $(\tau^\vee)^\vee = \tau$. Thus, $\Sigma_{\rm sat} = \sigma^\vee \cap \Z^m$, where $\sigma = \tau^{\vee}$.

If $\Lambda$ is trivial, then $Y=\C^n$ and Theorems \ref{t:main-positive}, \ref{t:first-dichotomy}(a), and \ref{t:second-dichotomy}(a) are immediate consequences of the Cartan extension theorem.  Therefore we will assume that $\Lambda$ is not trivial.  (If $\Lambda$ is trivial, the matrix $C$, introduced in the proof of Lemma \ref{l:trivial-kernel}, does not make sense.)

\begin{proof}[Proof of Theorem \ref{t:first-dichotomy}(a)]  Suppose that the normalisation of $Y$ is $\C^d$.  A nondegenerate holomorphic map $f$ from a normal subvariety $S$ of a reduced Stein space $X$ to $Y$ maps the complement of a thin subvariety of $S$ into $Y\cap(\C^*)^n$, which is smooth, so $f$ maps only a thin subvariety of $S$ into the non-normal locus of $Y$.  Hence $f$ lifts to a map into $\C^d$, which extends to a map from $X$ to $\C^d$ by the Cartan extension theorem.  Postcomposing by the normalisation map gives a holomorphic extension $X\to Y$ of $f$.
\end{proof}

Since $\C[x_1,\ldots,x_n]$ is Noetherian, $I$ is generated by finitely many binomials of the form $x^{\ell^+}-x^{\ell^-}$, where $\ell\in \Lambda$.  Choose once and for all a $k\times n$ matrix $A=(a_{\nu j})$ with integer entries, such that the binomials $x^{\ell^+}-x^{\ell^-}$, where $\ell$ runs through the rows of $A$, generate $I$.  Then a nondegenerate holomorphic map $g$ into $\C^n$ maps into $Y$ if and only if the well-defined meromorphic functions $g_1^{a_{\nu 1}}\cdots g_n^{a_{\nu n}}$, $\nu=1,\ldots,k$, are identically equal to $1$.

We further assume that $Y$ contains the origin $0$ in $\C^n$.  Equivalently, $\ker A$ contains a vector $(v_1,\ldots,v_n)$ with strictly positive integer entries (theorem of Stiemke \cite[Corollary 7.1k]{Schrijver1986}).  Then $Y$ is contractible (even holomorphically contractible), as shown by the homotopy
\[ Y\times [0,1]\to Y, \quad (y,t)\mapsto (t^{v_1}y_1,\ldots,t^{v_n}y_n), \]
so there is no obstruction to extending a continuous map into $Y$ from a subcomplex of a CW complex to the whole complex.  If $S$ is a subvariety of a reduced complex space $X$, then there is a triangulation of $X$ that restricts to a triangulation of $S$.  Hence every continuous map $S\to Y$ has a continuous extension $X\to Y$.

The nontrivial submonoid $K=\ker A\cap\N^n$ of $\N^n$ has a Hilbert basis.  This is a finite subset $\mathscr B$ of $K$ such that every vector in $K$ can be written as a linear combination of the vectors in $\mathscr B$ with coefficients in $\N$, and $\mathscr B$ is uniquely determined as the smallest subset of $K$ with this property.  (For a proof, see \cite[Theorem 8.2.9]{CLO2005} or \cite[Algorithm 1.4.5]{Sturmfels1993}, where an algorithm for finding the Hilbert basis is given.)  Let $B=(b_{ji})$ be the $n\times m$ matrix whose columns are the vectors of $\mathscr B$.  Then $AB=0$.  

\begin{lemma}  \label{l:trivial-kernel}
If $\ker B\subset\Z^m$ is trivial, then the normalisation of $Y$ is $\C^m$.
\end{lemma}

\begin{proof}
Let $C$ be a $p\times n$ matrix whose rows form a $\Z$-basis of $\Lambda$.  Then $CB=0$.  Consider the short exact sequence
\[ \xymatrix{ 0 \ar[r] & \Z^p \ar[r]^{C^t} & \Z^n \ar[r] & \Z^n/\Lambda \ar [r] & 0.} \]
It is exact at $\Z^p$ because the columns of $C^t$ are linearly independent.  It splits since $\Z^n/\Lambda$ is free, so $C^t$ has a left inverse.  Hence $C$ has a right inverse.   Since the columns of $B$ form the Hilbert basis of $\ker A\cap \N^n$, we have $\ker A\cap\N^n=B(\N^m)$.  Since $\ker A$ contains a vector with strictly positive integer entries, it follows that $\ker A=\im B\subset\Z^n$.  Also, $\ker C=\ker A$ since $A$ and $C$ have the same row space $\Lambda$.  Therefore the sequence
\[ \xymatrix{ 0 \ar[r] & \Z^m \ar[r]^B & \Z^n \ar[r]^C & \Z^p \ar [r] & 0} \]
is exact.  The dual sequence
\[ \xymatrix{ 0 \ar[r] & \Z^p \ar[r]^{C^t} & \Z^n \ar[r]^{B^t} & \Z^m \ar [r] & 0} \]
is also exact, so $\ker B^t=\Lambda$.  

Hence $Y$ is defined by the affine semigroup $\Sigma = B^t(\N^n) \subset \Z^m$ spanned over $\N$ by the rows $b_1,\ldots,b_n$ of $B$.  The normalisation of $Y$ is given by the semigroup $\tau\cap\Z^m$, where $\tau$ is the rational polyhedral cone $\Q_+ b_1+\cdots+\Q_+ b_n$. 

Let us determine the dual cone $\sigma=\tau^\vee=\{v\in\Q^m:Bv\in\Q_+^n\}$.  Clearly, $\Q_+^m\subset\sigma$.  Conversely, say $v\in\sigma\cap\Z^m$.  Then $Bv\in\N^n$ and $ABv=0$, so by the Hilbert basis property, there is $w\in\N^m$ such that $Bv=Bw$.  Since $B$ is injective, $v=w\in\N^m$.  Thus $\sigma\cap\Z^m=\N^m$, so $\sigma\subset\Q_+^m$, and $\sigma=\Q_+^m$.  Hence the normalisation of $Y$ is defined by the affine semigroup $\tau \cap \Z^m = \sigma^\vee \cap \Z^m = \N^m$, so the normalisation is $\C^m$.
\end{proof}

We now consider the case when the normalisation of $Y$ is not $\C^d$, so $\ker B\subset\Z^m$ is nontrivial.  We choose once and for all an $m\times\ell$ matrix $E=(e_{i\lambda})$ with integer entries whose columns form a $\Z$-basis for $\ker B\subset\Z^m$.  Then $BE=0$.

The easiest way to extend a holomorphic map $f:S\to Y$ to a holomorphic map $X\to Y$ would be to find holomorphic functions $g_1,\ldots,g_m$ on $S$ such that
\[ f_j=g_1^{b_{j1}}\cdots g_m^{b_{jm}} \]
for $j=1,\ldots,n$, and extend $g_1,\ldots,g_m$ to holomorphic functions $G_1,\ldots,G_m$ on $X$.  Setting
\[ F_j=G_1^{b_{j1}}\cdots G_m^{b_{jm}}, \]
we get a holomorphic map $F:X\to Y$ extending $f$.  Such a factorisation of $f_1,\ldots,f_n$ need not exist.  However, a weaker kind of factorisation exists when $f$ is a nondegenerate map from a factorial space.

A \textit{twisted factorisation} of a holomorphic map $f:W\to Y$, where $W$ is a complex space, consists of line bundles $N_1,\ldots,N_\ell$ on $W$ and sections $\sigma_1,\ldots,\sigma_m$ of the bundles 
\[ L_i=N_1^{e_{i1}}\otimes\cdots\otimes N_\ell^{e_{i\ell}}, \qquad i=1,\ldots,m, \]
such that with respect to the canonical trivialisations of the bundles
\[ M_j = L_1^{b_{j1}}\otimes\cdots\otimes L_m^{b_{jm}}, \]
we have
\[ f_j=\sigma_1^{b_{j1}}\otimes\cdots\otimes \sigma_m^{b_{jm}}, \qquad j=1,\ldots,n. \]
(The above formulas may be abbreviated as $L=N^E$, $M=L^B$, and $f=\sigma^B$.)  Each bundle $M_j$ has a canonical trivialisation because it is a tensor product of the factors $N_\lambda^{b_{j1}e_{1\lambda}+\cdots+b_{jm}e_{m\lambda}}$, $\lambda=1,\ldots,\ell$, with $b_{j1}e_{1\lambda}+\cdots+b_{jm}e_{m\lambda}=0$.  Every choice of local trivialisations of $N_1,\ldots,N_\ell$ induces the same local trivialisation of $M_j$ and thus defines a global trivialisation of $M_j$.  In other words, every choice of defining cocycles for $N_1,\ldots,N_\ell$ induces the trivial cocycle as a defining cocycle for $M_j$.  In this way, a section of $M_j$ is identified with a function.

Conversely, if a nondegenerate holomorphic map $f:W\to\C^n$ has a twisted factorisation, then $f_1^{a_{\nu 1}}\cdots f_n^{a_{\nu n}}=1$ on $W$ for $\nu=1,\ldots,k$ since $AB=0$, so $f$ maps $W$ into~$Y$.

\begin{theorem}  \label{t:twisted-factorisation}
Let $W$ be a factorial complex space and $Y\subset \C^n$ be a nnn affine toric variety of dimension $d$ with $0\in Y$, whose normalisation is not $\C^d$.  Every nondegenerate holomorphic map $f:W\to Y$ has a twisted factorisation.
\end{theorem}

\begin{proof}
Write $((f_1),\ldots,(f_n))=\sum v_\gamma Z_\gamma$, where the hypersurfaces $Z_\gamma$ are the irreducible components of the zero sets of $f_1,\ldots,f_n$ and $v_\gamma\in\N^n$.  Since $f_1^{a_{\nu 1}}\cdots f_n^{a_{\nu n}}=1$ for $\nu=1,\ldots,k$, we have $v_\gamma\in\ker A$ for each $\gamma$.  Choose $u_\gamma\in\N^m$ with $B u_\gamma=v_\gamma$ and set $(D_1,\ldots,D_m)=\sum u_\gamma Z_\gamma$.

Since $W$ is factorial, every divisor on $W$ is locally principal (that is, every Weil divisor is a Cartier divisor).  Thus there is an open cover $(U_\alpha)$ of $W$ and holomorphic functions $g_{\alpha 1},\ldots,g_{\alpha m}$ on $U_\alpha$ such that $(g_{\alpha i})=D_i\vert U_\alpha$.  Then, for $j=1,\ldots,n$,
\[ (g_{\alpha 1}^{b_{j1}}\cdots g_{\alpha m}^{b_{jm}}) = \sum_i b_{ji} D_i = \sum_{i,\gamma} b_{ji} u_{\gamma i}Z_\gamma = \sum_\gamma v_{\gamma j} Z_\gamma = (f_j) \]
on $U_\alpha$.  For $i=1,\ldots,m$, the cocycle $(g_{\alpha i}/g_{\beta i})$ defines a line bundle $L_i$ on $W$.  The cochain $(g_{\alpha i})$ defines a section $\sigma_i$ of $L_i$.  For $j=1,\ldots,n$, let $M_j=L_1^{b_{j1}}\otimes\cdots\otimes L_m^{b_{jm}}$.  Then $M_j$ has the section $\sigma_1^{b_{j1}}\otimes\cdots\otimes\sigma_m^{b_{jm}}$ with divisor $(f_j)$, so $M_j$ is trivial.  

Since $M_1,\ldots,M_n$ are trivial, $L=(L_1,\ldots,L_m)$ is in the kernel of the map 
\[ H^2(W,\Z)\otimes\Z^m\to H^2(W,\Z)\otimes\im B \] 
defined by $B$.  Since the columns of $E$ are linearly independent, the sequence
\[ \xymatrix{ 0 \ar[r] & \Z^\ell \ar[r]^E & \Z^m \ar[r]^B & \im B \ar[r] & 0} \]
is exact, and since $\im B\subset\Z^n$ is free, it splits.  Thus the induced sequence
\[ \xymatrix{ 0 \ar[r] & H^2(W,\Z)\otimes\Z^\ell \ar[r]^E & H^2(W,\Z)\otimes\Z^m \ar[r]^B & H^2(W,\Z)\otimes\im B \ar[r] & 0} \]
is also exact, so $L=N^E$ for some $N=(N_1,\ldots,N_\ell)$.

Let $\eta_j=f_j^{-1}\sigma_1^{b_{j1}}\otimes\cdots\otimes\sigma_m^{b_{jm}}\in\mathscr O^*(W)$, $j=1,\ldots,n$.  We will show that there are $\xi_1,\ldots,\xi_m\in\mathscr O^*(W)$ such that $\eta_j=\xi_1^{b_{j1}}\cdots\xi_m^{b_{jm}}$, that is, $\eta=\xi^B$.  Then we simply replace $\sigma_i$ by $\xi_i^{-1}\sigma_i$.

As in the proof of Lemma \ref{l:trivial-kernel}, let $C=(c_{\mu j})$ be a $p\times n$ matrix whose rows form a $\Z$-basis of $\Lambda$.  Let the $n\times p$ matrix $Q=(q_{j\mu})$ with integer entries be a right inverse for $C$.  Then $C(QC-I_n)=0$.  Since $\ker C=\im B\subset\Z^n$, there is an $m\times n$ matrix $P$ with integer entries such that $QC-I_n=BP$.  

Let $\xi=\eta^{-P}$.  Then $\eta=\eta^{QC-BP}=\xi^B\cdot(\eta^C)^Q$, where the dot denotes componentwise multiplication.  We have $f^C=(1,\ldots,1)$ since $f$ maps into $Y$, and $(\sigma_1^{b_{j1}}\otimes\cdots\otimes\sigma_m^{b_{jm}})_{j=1,\ldots,n}^C=(1,\ldots,1)$ since $CB=0$, so $\eta^C=(1,\ldots,1)$ and $\eta=\xi^B$.
\end{proof}

The following corollary is immediate.

\begin{corollary}  \label{c:untwisted-factorisation}
Let $W$ be a factorial complex space on which every holomorphic line bundle is trivial.  Let $Y\subset \C^n$ be a nnn affine toric variety of dimension $d$ with $0\in Y$, whose normalisation is not $\C^d$.  For every nondegenerate holomorphic map $f:W\to Y$, there is a holomorphic map $g:W\to\C^m$ such that $f=g^B$.
\end{corollary}

The next theorem characterises extendability of nondegenerate maps.

\begin{theorem}  \label{t:characterisation}
Let $S$ be a factorial subvariety of a factorial Stein space $X$.  Let $Y\subset \C^n$ be a nnn affine toric variety of dimension $d$ with $0\in Y$, whose normalisation is not $\C^d$, and let $f:S\to Y$ be a nondegenerate holomorphic map.  Then $f$ extends to a holomorphic map $X\to Y$ if and only if $f$ has a twisted factorisation such that $N_1,\ldots,N_\ell$ extend to line bundles on $X$.
\end{theorem}

\begin{proof}
Suppose that $f$ has a twisted factorisation such that $N_1,\ldots,N_\ell$ extend to line bundles $\tilde N_1,\ldots,\tilde N_\ell$ on $X$.  Denote the corresponding extension of $L_i$ by $\tilde L_i$.  Extend $\sigma_i$ to a section $\tilde\sigma_i$ of $\tilde L_i$.  We obtain a holomorphic extension $F:X\to\C^n$ of $f$ with $F_j=\tilde\sigma_1^{b_{j1}} \otimes\cdots\otimes \tilde\sigma_m^{b_{jm}}$.  Because $F$ has a twisted factorisation, $F(X)\subset Y$.

Conversely, if $f:S\to Y$ is nondegenerate and extends to a holomorphic map $F:X\to Y$, then $F$ is nondegenerate on the connected components of $X$ that intersect $S$.  In addition, by possibly changing $F$ (to, say, the constant nondegenerate map $x \mapsto (1,\ldots ,1) \in Y$) on any connected component of $X$ that does not intersect $S$, we may assume that $F$ is nondegenerate.  Now take a twisted factorisation for $F$ and restrict it to $S$.
\end{proof}

The following corollary gives Theorem \ref{t:main-positive} under the additional assumptions that $f$ is nondegenerate and $0\in Y$.

\begin{corollary}  \label{c:sufficient-for-nondegenerate}
Let $S$ be a factorial subvariety of a reduced Stein space $X$.  Let $Y\subset \C^n$ be a nnn affine toric variety with $0\in Y$, and let $f:S\to Y$ be a nondegenerate holomorphic map.  If the restriction map $H^2(X,\Z)\to H^2(S,\Z)$ is surjective, in particular if $H^2(S,\Z)=0$, then $f$ extends to a holomorphic map $X\to Y$.
\end{corollary}

\begin{proof}
If the normalisation of $Y$ is $\C^d$, then the corollary follows from Theorem \ref{t:first-dichotomy}(a).  If the normalisation of $Y$ is not $\C^d$, then the corollary follows from Theorem \ref{t:twisted-factorisation} and the proof of Theorem \ref{t:characterisation}, where factoriality of $X$ was not needed for the implication that is relevant here.
\end{proof}

We can reformulate the characterisation of extendibility of nondegenerate maps without reference to twisted factorisations as follows.

\begin{corollary}  \label{c:reformulation}
Let $S$ be a factorial subvariety of a factorial Stein space $X$.  Let $Y\subset \C^n$ be a nnn affine toric variety of dimension $d$ with $0\in Y$, whose normalisation is not $\C^d$, and let $f:S\to Y$ be a nondegenerate holomorphic map.  Then $f$ extends to a holomorphic map $X\to Y$ if and only if there is an $m$-tuple $D$ of effective divisors on $S$ such that $BD=(f)$ and $c_1(D)\in H^2(S,\Z)^m$ is the restriction of an element of $\ker B\subset H^2(X,\Z)^m$.
\end{corollary}

The main point of the proof is that if $D$ exists, then a twisted factorisation of $f$ with $(\sigma)=D$ is obtained as in the proof of Theorem \ref{t:twisted-factorisation}.

Next we prove Theorem \ref{t:first-dichotomy}(b).  We need the following two lemmas.

\begin{lemma}  \label{l:annulus}
The domain
\[ S=\{z\in\C^2:\tfrac 3 4<\lvert z_i\rvert<\tfrac 5 4,\ i=1,2\} \]
contains mutually disjoint smooth connected curves $W_1$, $W_2$, $Z_1$, $Z_2$ such that
\[ c_1(W_1)=c_1(W_2)=-c_1(Z_1)=-c_1(Z_2)\neq 0 \]
in $H^2(S,\Z)=\Z$. 
\end{lemma}

\begin{proof}
As shown in \cite[Section VI.5.2]{Range1986}, if we only wanted disjoint divisors $W$ and $Z$ on $S$ such that $c_1(W)=-c_1(Z)\neq 0$, we could decompose the principal divisor $D=\{z\in S:z_1-z_2=1\}$ into the divisors $W=\{z\in D:\im z_1>0\}$ and $Z=\{z\in D:\im z_1<0\}$, which are not principal.  In the same way, the principal divisor $D_\epsilon=\{z\in S:z_1-z_2=1+\epsilon\}$, $-\tfrac 1 2<\epsilon<\tfrac 1 2$, splits into disjoint nonprincipal divisors $D_\epsilon^+$ and $D_\epsilon^-$, which are easily seen to be connected (look at the intersection of the annulus $\{\zeta\in\C:\tfrac 3 4<\lvert \zeta \rvert<\tfrac 5 4\}$ with itself translated to the left by $1+\epsilon$).  Since $H^2(S,\Z)$ is countable, there are distinct $\epsilon_1$ and $\epsilon_2$ such that $c_1(D_{\epsilon_1}^+)=c_1(D_{\epsilon_2}^+)$.  Then $W_1=D_{\epsilon_1}^+$, $W_2=D_{\epsilon_2}^+$, $Z_1=D_{\epsilon_1}^-$, and $Z_2=D_{\epsilon_2}^-$ have the desired properties.
\end{proof}

\begin{lemma}  \label{l:minimal}
Let $A$ be a $k\times n$ matrix with integer entries such that $K=\ker A\cap\N^n$ contains a vector with strictly positive entries.  Let $B$ be the $n\times m$ matrix whose columns form the Hilbert basis of $K$.  Suppose that $\ker B\subset\Z^m$ is not trivial.  Then there is a vector $v\in K\setminus\{0\}$ such that $v=Bw=Bz$ for two distinct vectors $w, z\in\N^m$, and $v$ is minimal in the sense that if $v'\in K\setminus\{0\}$, $v'\leq v$, and $v'\neq v$, then there is at most one $w'\in\N^m$ with $Bw'=v'$.  Moreover, if $w,z\in\N^m$ are distinct and $v=Bw=Bz$, then $\supp w\cap\supp z=\varnothing$.
\end{lemma}

\begin{proof}
Since $K$ contains a vector with strictly positive entries, no row of $B$ consists of zeros only, so $\ker B\cap\N^m=\{0\}$.  Let $u\in\ker B\subset\Z^m$, $u\neq 0$.  Then $B u^+=B u^-\in K\setminus\{0\}$.  Thus there is a vector $v\in K\setminus\{0\}$ such that $v=Bw=Bz$ for two distinct vectors $w, z\in\N^m$.  Take $v$ to be a minimal such vector.

Assume $w,z\in\N^m$ are distinct, $v=Bw=Bz$, and $\supp w\cap\supp z\neq\varnothing$.  Then there is a vector $e\in\N^m$ of the form $(\ldots,0,1,0,\ldots)$ such that $e\leq w$ and $e\leq z$.  Now $w\neq e$, for otherwise $B w$ would be a column of $B$ that was expressible as the $\N$-linear combination $B z$ of the columns of $B$, contradicting the Hilbert basis property of the columns of $B$.  Similarly, $z\neq e$, so $w'=w-e$ and $z'=z-e$ are distinct elements of $\N^m\setminus\{0\}$. Then the vector $B w'=B z'$ contradicts the minimality of $v$.
\end{proof}

\begin{proof}[Proof of Theorem \ref{t:first-dichotomy}(b)]
By Lemma \ref{l:trivial-kernel}, $\ker B\subset\Z^m$ is not trivial.  Let $S$, $W_1$, $W_2$, $Z_1$, $Z_2$ be as in Lemma \ref{l:annulus}.  Since every annulus embeds into $\C^2$ \cite{Laufer1973}, $S$ embeds into $\C^4$.  Let $v\neq 0$ in $\ker A\cap\N^n$ and $w\neq z$ in $\N^m$ with $v=Bw=Bz$ be as in Lemma \ref{l:minimal}.  As in the proof of the lemma, we see that $w$ and $z$ are not of the form $(\ldots,0,1,0,\ldots)$.  Hence there are $w_1, w_2, z_1, z_2\in\N^m\setminus\{0\}$ with $w=w_1+w_2$ and $z=z_1+z_2$.

Consider the principal vector-valued divisor $D=BD'$ on $S$, where $D'=w_1 W_1+w_2 W_2 +z_1 Z_1 +z_2 Z_2$.  Find line bundles $L_1,\ldots,L_m$ on $S$ with sections $\sigma_1,\ldots,\sigma_m$ such that $(\sigma)=D'$.  Since $B(\sigma)$ is principal, $L^B$ is trivial, so $L=N^E$ for some $N$ (see the proof of Theorem \ref{t:twisted-factorisation}) and $\sigma^B$ is identified with a holomorphic map $f:S\to\C^n$.  Having a twisted factorisation, $f$ maps $S$ into $Y$.

We claim that the nondegenerate holomorphic map $f:S\to Y$ does not extend to a holomorphic map $\C^4\to Y$.  Suppose it does.  Then, by Corollary \ref{c:untwisted-factorisation}, there are $g_1,\ldots,g_m\in\O(S)$ such that $f_j=g_1^{b_{j1}}\cdots g_m^{b_{jm}}$ for $j=1,\ldots,n$.  Since $(f_j)$ is an $\N$-linear combination of the irreducible curves $W_1$, $W_2$, $Z_1$, $Z_2$, so is each $(g_i)$.  Moreover, since $c_1(W_1)=c_1(W_2)=-c_1(Z_1)=-c_1(Z_2)$, each $(g_i)$ must be an $\N$-linear combination of $W_1+Z_1$, $W_1+Z_2$, $W_2+Z_1$, $W_2+Z_2$.  Hence there are $\alpha_1,\alpha_2,\alpha_3,\alpha_4\in\N^m$ such that
\[ (g)=\alpha_1(W_1+Z_1)+\alpha_2(W_1+Z_2)+\alpha_3(W_2+Z_1)+\alpha_4(W_2+Z_2), \]
so
\[ (f)=B(g)=B(\alpha_1+ \alpha_2)W_1+B(\alpha_3+\alpha_4)W_2+B(\alpha_1+\alpha_3)Z_1+B(\alpha_2+\alpha_4)Z_2. \]
Since $(f)=D$, it follows that
\[ B(\alpha_1+ \alpha_2)=Bw_1, \ B(\alpha_3+\alpha_4)=Bw_2, \ B(\alpha_1+\alpha_3)=Bz_1, \ B(\alpha_2+\alpha_4)=Bz_2. \]
Now let $\alpha=\alpha_1+\alpha_2+\alpha_3+\alpha_4$.  Then $B\alpha=Bw=Bz=v$.  By Lemma \ref{l:minimal}, $\supp\alpha\cap\supp w=\varnothing$ or $\supp\alpha\cap\supp z=\varnothing$.  Say the latter holds.  Then $\alpha_1+\alpha_3\neq z_1$.  Let $v'=Bz_1\leq v$.  Then $v'\neq 0$ and $v'\neq v$ because $Bz_1\neq 0$ and $Bz_2\neq 0$.  Also, $v'=B(\alpha_1+\alpha_3)$, contradicting the minimality of $v$.
\end{proof}

By Lemma \ref{l:trivial-kernel}, if $\ker B$ is trivial, then the normalisation of $Y$ is $\C^m$, so $m=d$.  By Theorem \ref{t:first-dichotomy}(a) and the proof of Theorem \ref{t:first-dichotomy}(b), if $\ker B$ is not trivial, then the normalisation of $Y$ is not $\C^d$.

Now we prove Theorem \ref{t:second-dichotomy}.

\begin{proof}[Proof of Theorem \ref{t:second-dichotomy}]
(a)  Suppose that $Y$ is locally irreducible.  Then $\C^d$ is also the seminormalisation of $Y$.  Let $S$ be a seminormal subvariety of a reduced Stein space $X$, and let $f:S\to Y$ be a holomorphic map.  By functoriality of seminormalisation, $f$ lifts to a map $S\to \C^d$, which extends to a map $X\to\C^d$.  Postcomposing by the normalisation map gives a holomorphic extension $X\to Y$ of $f$.

(b)  Suppose that $Y$ is not locally irreducible.  Let $\phi:\C^m\to Y$, $t\mapsto t^B$.  As shown in the proof of Lemma \ref{l:trivial-kernel}, the saturation of the semigroup in $\Z^m$ spanned over $\N$ by the rows of $B$ is $\N^m$.  Hence, for each $i=1,\ldots,m$, $B$ has a row of the form $(\ldots,0,N_i,0,\ldots)$ with $N_i\geq 1$ in the $i^\textrm{th}$ place, so the monomial $t_i^{N_i}$ is a component of $\phi$.  Hence $\phi$ is proper and finite.  Thus the image of $\phi$ is a subvariety of $Y$ of dimension $m$.  Since $Y$ itself has dimension $m$, it follows that $\phi$ is surjective.

We claim that $\phi\vert\phi^{-1}((\C^*)^n)$ is injective.  It then follows that $\phi$ is the normalisation map of $Y$.  Since $t_1^{N_1},\ldots,t_m^{N_m}$ are components of $\phi$, $\phi^{-1}((\C^*)^n)=(\C^*)^m$.  Because the components of $\phi$ are monomials, it suffices to show that if $\phi(t)=(1,\ldots,1)\in Y\subset\C^n$, then $t=(1,\ldots,1)\in\C^m$.  Write $t=\exp(2\pi i v)$ with $v\in\C^m$ and set $w=Bv$.  Then $t^B=(1,\ldots,1)$ implies that $w\in\Z^n$.  Also, $Cw=CBv=0$, so there is $v'\in\Z^m$ with $w=Bv'$.  Since $B$ is injective, $v=v'$ and $t=\exp(2\pi iv')=(1,\ldots,1)$, proving our claim.

For $i=1,\ldots,m$, let $\phi^i:\C\to Y$, $s\mapsto\phi(\ldots,0,s,0,\ldots)$, with $s$ in the $i^\textrm{th}$ place.  Each component of $\phi^i$ is a power of $s$ or $0$.  Since the components of $\phi$ are monomials, if $\phi^i$ is injective for all $i$, then $\phi$ is injective.  Since $Y$ is not locally irreducible, its normalisation map $\phi$ is not injective, so $\phi^i$ is not injective for some $i$.  Write $\phi^i(s)=\psi(s^k)$ with $k\geq 2$ and $\psi=(\psi_1,\ldots,\psi_n):\C\to Y$ injective.  

Consider the holomorphic map $f:\C\times\{0\}\cup\{(0,1)\}\to Y$, $f(s,0)=\psi(s)$, $f(0,1)=(1,\ldots,1)$.  Note that $f$ is degenerate.  Namely, if none of the components of $f(\cdot,0)=\psi$ is identically zero, then $m=1$ because $B$ does not have a column of zeros, but then $\phi$ is injective.

Suppose that $f$ extends to a holomorphic map $F:\C^2\to Y$.  Then $F$ is nondegenerate, so by Corollary \ref{c:untwisted-factorisation}, there are $g_1,\ldots,g_m\in\O(\C)$ such that $\psi=F(\cdot,0)=g^B=\phi(g_1,\ldots,g_m)$ on $\C$, that is, 
\[ \phi(g_1(s^k),\ldots,g_m(s^k)) = \psi(s^k) = \phi(\ldots,0,s,0,\ldots) \]
for all $s\in\C$.  Recall that $t_i^{N_i}$ is a component of $\phi$.  It follows that $g_i(s^k)^{N_i}=s^{N_i}$ for all $s\in\C$, which is absurd.
\end{proof}

Note that we have in fact shown that the degenerate holomorphic map $\C\times\{0\}\to Y$, $(s,0)\mapsto\psi(s)$, does not extend to a nondegenerate holomorphic map $\C^2\to Y$.

Taking $Y$ to be Whitney's umbrella, defined by the equation $x^2 y=z^2$ in $\C^3$, we see that the source $\C\times\{0\}\cup\{(0,1)\}$ must be disconnected in general.  A degenerate holomorphic map $f$ into $Y$ from a connected source $S\subset X$ maps into either $Y\cap\{x=z=0\}$ or $Y\cap\{y=z=0\}$, both of which are biholomorphic to $\C$, so $f$ extends to $X$.

To complete the proof of Theorem \ref{t:main-positive}, we need the following lemma.

\begin{lemma}  \label{l:make-nondegenerate}
Let $Y$ be a nnn affine toric variety.  Let $W$ be an irreducible reduced complex space.  If $f:W\to Y$ is a holomorphic map, then there exists a nnn toric subvariety $Z$ of $Y$ such that $f$ is a nondegenerate map into $Z$.
\end{lemma}

\begin{proof}
Now $Y$ is partitioned into finitely many torus orbits.  Hence $W$ is partitioned into the preimages of finitely many torus orbits.  The Hausdorff closure in $W$ of one of the preimages must have nonempty interior.  Then, by the identity theorem, $f(W)$ lies in the Zariski closure $Z$ of the corresponding torus orbit $O$, and $f(W)$ intersects $O$.  Now $Z$ is a nnn toric subvariety of $Y$ whose torus is $O$ \cite[page 8]{Mustata2004}, so it follows that $f$ is nondegenerate as a map into $Z$.
\end{proof}

\begin{proof}[Proof of Theorem \ref{t:main-positive}]
We may assume that $S$ is connected and hence irreducible.  By Lemma \ref{l:make-nondegenerate}, we may also assume that $f$ is nondegenerate.

Since $f$ is nondegenerate, it maps the complement of a thin subvariety of $S$ into $Y\cap(\C^*)^n$, which is smooth, so $f$ maps only a thin subvariety of $S$ into the non-normal locus of $Y$.  Hence $f$ lifts to a map $g$ into the normalisation $\tilde Y$ of $Y$.  By the structure theory of normal affine toric varieties, $\tilde Y=Y'\times (\C^*)^r$, where $r\geq 0$ and $Y'$ is a normal affine toric variety embeddable in $\C^s$ in such a way that $0\in Y'$.  Since $H^1(X,\Z)\to H^1(S,\Z)$ is surjective, the map $S\stackrel{g}{\to}\tilde Y\to(\C^*)^r$ extends to $X$.  

We need to show that the map $g':S\stackrel{g}{\to}\tilde Y\to Y'$ is nondegenerate.  Then the proof is complete by Corollary \ref{c:sufficient-for-nondegenerate}.  In other words, we need the image of $g'$ to intersect the torus in $Y'$.  It does because the image of $f$ intersects the torus in $Y$, the normalisation map from $\tilde Y$ to $Y$ restricts to a biholomorphism between their respective tori, and the torus in $\tilde Y$ is the product of $(\C^*)^r$ and the torus in $Y'$.
\end{proof}

\begin{example}  \label{x:cone}
We will now illustrate our results by one of the simplest examples of a normal affine toric variety, the cone $Y = \{z\in\C^3:z_1 z_2 = z_3^2\}$, which has a single, normal singularity at the origin $0$.  It is interesting to note that $Y$ is very close to being an Oka manifold.  It has the structure of a line bundle of degree $-2$ over the projective line with the zero section blown down.  Hence $Y$ with $0$ removed is an Oka manifold.  Likewise, $Y$ with $0$ blown up is an Oka manifold.  Moreover, $Y$ is the quotient of $\C^2$ by the action of $\Z_2$ that identifies $(z,w)$ with $(-z,-w)$, so $Y$ has the Oka manifold $\C^2$ as a two-sheeted covering space branched over a single point.

The matrices $A$, $B$, and $E$ associated to $Y$ are
\[ A = \left[ \begin{array}{ccc} 1 & 1 & -2 \end{array} \right], \quad
   B = \left[ \begin{array}{ccc} 2 & 0 & 1 \\ 0 & 2 & 1 \\ 1 & 1 & 1 \end{array} \right], \quad
   E = \left[ \begin{array}{c}   1 \\ 1 \\ -2 \end{array} \right].  \]
A twisted factorisation of a holomorphic map $f:S \to Y$, where $S$ is a factorial subvariety of a factorial Stein space $X$, involves sections $\sigma_1$, $\sigma_2$, $\sigma_3$ of line bundles $L_1$, $L_2$, $L_3$.  Since $M = L^B = (L_1^2 L_3, L_2^2 L_3, L_1 L_2 L_3)$ is a triple of trivial line bundles, $L = N^E$ for some line bundle $N$, that is, $(L_1, L_2, L_3) = (N, N, N^{-2})$.  A local frame $e$ for $N$ induces the frames $e$, $e$, $e^{-2}$ for $L_1$, $L_2$, $L_3$, which in turn induce the frames $e^2 e^{-2}$, $e^2 e^{-2}$, $e e e^{-2}$ for $M_1$, $M_2$, $M_3$, which clearly are independent of the choice of $e$.  With respect to these frames, the twisted factorisation is of the form
\[  f = (\sigma_1^2 \sigma_3, \sigma_2^2 \sigma_3, \sigma_1 \sigma_2 \sigma_3). \]
By Theorem \ref{t:twisted-factorisation}, every nondegenerate holomorphic map $f:S\to Y$ has such a factorisation, which is in fact easy to see directly in this case.  Namely, the divisor of $f$ has the form $(2 D_1 + D_3, 2D_2 + D_3, D_1 + D_2 + D_3)$, where $D_1$, $D_2$, $D_3$ are effective divisors.  For $i=1, 2, 3$, let $L_i$ be the line bundle defined by $D_i$, and let $\sigma_i$ be a section of $L_i$ with divisor $D_i$.  We then obtain a twisted factorisation of $f$ by multiplying $\sigma_1$, $\sigma_2$, $\sigma_3$ by suitable invertible holomorphic functions.  By Theorem \ref{t:characterisation}, $f$ extends to a map $X \to Y$ if and only if there is a twisted factorisation such that the line bundle $N$ extends to $X$.

Note that only when the common divisors of the components of $f$ have multiplicity at most $1$ is the decomposition $(f) = (2 D_1 + D_3, 2D_2 + D_3, D_1 + D_2 + D_3)$ unique, so usually there are different choices of the line bundle $N$, and it might be the case that some choices of $N$ extend to $X$ and others do not.

By Theorem \ref{t:first-dichotomy}(b), there is a smooth surface $S$ in $\C^4$, biholomorphic to a product of two annuli, and a nondegenerate map $f :S\to Y$ which does not extend to $\C^4$.  Let us describe the construction of such a map more concretely.  The vectors $v$, $w$, $z$ in Lemma \ref{l:minimal} are unique, and they are 
\[ v = \left[ \begin{array}{c}   2 \\ 2 \\ 2 \end{array} \right], \quad
   w = \left[ \begin{array}{c}   1 \\ 1 \\ 0 \end{array} \right], \quad
   z = \left[ \begin{array}{c}   0 \\ 0 \\ 2 \end{array} \right]. \]
Let $D = (2Z_1 + W_1 + W_2, 2Z_2 + W_1 + W_2, Z_1 + Z_2 + W_1 + W_2)$, where $Z_1$, $Z_2$, $W_1$, $W_2$ are as in Lemma \ref{l:annulus}.  Since $c_1(D) = 0$, there is a triple of holomorphic functions $(h_1, h_2, h_3)$ with divisor $D$.  Multiplying $h_1$ by the invertible function $h_3^2/(h_1 h_2)$, we get a holomorphic map $f=(h_3^2/h_2, h_2, h_3):S\to Y$.  We claim that $f$ does not extend to $\C^4$.  If it does, then it factors as $g^B$ by Corollary \ref{c:untwisted-factorisation}.  Then $(g)=(Z_1, Z_2, W_1 + W_2)$, so $Z_1$, $Z_2$, $W_1+W_2$ are principal divisors, contradicting Lemma \ref{l:annulus}.
\end{example}

\begin{proof}[Proof of Proposition \ref{p:approximation}]
The proof is an easy adaptation of the proof of \cite[Theorem 1]{Larusson2005}.  For the reader's convenience we provide the details.  Let $\phi:\Omega\to\C^m$ be a holomorphic embedding (we take a Runge domain to be Stein by definition).  The inclusion $i:\Omega\hookrightarrow X$ factors through the Stein space $M=X\times\C^m$ as $\Omega\stackrel{j}{\to} M \stackrel{\pi}{\to} X$, where $j=(i,\phi)$ is an embedding and $\pi$ is the projection.

By Theorem \ref{t:main-positive}, the holomorphic map $f\circ\pi:j(\Omega)\to Y$ extends to a holomorphic map $h:M\to Y$.  Since $\Omega$ is Runge, we can approximate $\phi:\Omega\to\C^m$ uniformly on compact subsets of $\Omega$ by holomorphic maps $\psi:X\to\C^m$.  Then the maps $h\circ (\id_X, \psi):X\to Y$ are holomorphic and approximate $f$.
\end{proof}

\begin{proof}[Proof of Proposition \ref{p:discrete}]
There is a holomorphic surjection $h:\C^r\to Y$ \cite[Theorem 3]{KT2003}.  Lift the given map $S\to Y$ by $h$, extend the lifting to a holomorphic map $X\to \C^r$ using the Cartan extension theorem, and postcompose by $h$.
\end{proof}

\begin{proof}[Proof of Proposition \ref{p:factorial-implies-smooth}]
Suppose that $Y$ is a factorial affine toric variety.  Then $Y$ is normal, so $Y=Y'\times (\C^*)^r$, where $r\geq 0$ and $Y'$ is an affine toric variety embeddable in $\C^s$ in such a way that $0\in Y'$.  Since $Y$ is factorial, so is $Y'$.  Apply Theorem \ref{t:main-positive} with $X=\C^s$, $S=Y'$, and $f=\id_{Y'}$, using the fact that $Y'$ is contractible.  We conclude that $Y'$ is a holomorphic retract of $\C^s$ and is therefore smooth, so $Y$ is smooth.
\end{proof}

\begin{proof}[Proof of Proposition \ref{p:meromorphic}]
If the normalisation of $Y$ is $\C^d$, then the proposition follows from Theorem \ref{t:first-dichotomy}(a).  Suppose that the normalisation of $Y$ is not $\C^d$.  Using Theorem \ref{t:twisted-factorisation}, choose a twisted factorisation $f_j=\sigma_1^{b_{j1}}\otimes\cdots\otimes \sigma_m^{b_{jm}}$ of $f$ on $S$ with bundles $N_\lambda$, $L_i$, and $M_j$ on $S$ as above.  Let $t_\lambda$ be a nontrivial holomorphic section of $N_\lambda$ over $S$.  Define the meromorphic section $\tau_i=t_1^{e_{i1}}\otimes\cdots\otimes t_\ell^{e_{i\ell}}$ of $L_i$.  Then the meromorphic section $\tau_1^{b_{j1}}\otimes\cdots\otimes \tau_m^{b_{jm}}$ of $M_j$ is identically equal to $1$ in the canonical trivialisation, so the meromorphic functions $g_i=\sigma_i/\tau_i$ satisfy $f_j=g_1^{b_{j1}}\cdots g_m^{b_{jm}}$.  Extend $g_i$ to a meromorphic function $G_i$ on $X$.  Then $(F_1,\ldots,F_n)$, where $F_j=G_1^{b_{j1}}\cdots G_m^{b_{jm}}$, is a meromorphic extension of $f$.
\end{proof}


\begin{thebibliography}{88}

\bibitem{CLO2005}
Cox, D. A., J. Little, D. O'Shea. \textit{Using algebraic geometry.}  Second edition.  Graduate Texts in Mathematics, 185.  Springer, 2005.

\bibitem{CLS2011}
Cox, D. A., J. B. Little, H. K. Schenck.  \textit{Toric varieties.}  Graduate Studies in Mathematics, 124. American Mathematical Society, 2011.

\bibitem{Forstneric2009}
Forstneri\v c, F.  \textit{Oka manifolds.}  C. R. Math. Acad. Sci. Paris \textbf{347} (2009) 1017--1020.

\bibitem{Forstneric2011}
Forstneri\v c, F.  \textit{Stein manifolds and holomorphic mappings.  The homotopy principle in complex analysis.}  Ergebnisse der Mathematik und ihrer Grenzgebiete, 3.\ Folge, 56.  Springer, 2011.

\bibitem{Forstneric2013}
Forstneri\v c, F.  \textit{Oka manifolds: from Oka to Stein and back.}  With an appendix by F.\ L\'arusson.  Ann. Fac. Sci. Toulouse Math. (6) \textbf{22} (2013) 747--809.

\bibitem{FL2011} 
Forstneri\v c, F., F. L\'arusson.  \textit{Survey of Oka theory.}  New York J. Math. \textbf{17a} (2011) 1--28. 

\bibitem{Gromov1989}
Gromov, M.  \textit{Oka's principle for holomorphic sections of elliptic bundles.}  J. Amer. Math. Soc. \textbf{2} (1989) 851--897.

\bibitem{KT2003}
Katsabekis, A., A. Thoma.  \textit{Toric sets and orbits on toric varieties.}  J. Pure Appl. Algebra \textbf{181} (2003) 75--83.

\bibitem{Larusson2005}
L\'arusson, F.  \textit{Mapping cylinders and the Oka principle.}  Indiana Univ. Math. J. \textbf{54} (2005) 1145--1159.

\bibitem{Larusson2011}
L\'arusson, F.  \textit{Smooth toric varieties are Oka.}  \texttt{arXiv:1107.3604}

\bibitem{Laufer1973}
Laufer, H. B.  \textit{Imbedding annuli in ${\bf C}^2$.}  J. Analyse Math. \textbf{26} (1973) 187--215. 

\bibitem{MS2005}
Miller, E., B. Sturmfels.  \textit{Combinatorial commutative algebra.}  Graduate Texts in Mathematics, 227.  Springer, 2005.

\bibitem{Mustata2004}
Mustata, M.  \textit{Semigroups and affine toric varieties.}  Chapter 1 of lecture notes on toric varieties.  Accessed at {\tt http://www.math.lsa.umich.edu/$\sim$mmustata/toric\textunderscore var.html} on 28 October 2014.

\bibitem{Range1986}
Range, R. M.  \textit{Holomorphic functions and integral representations in several complex variables.}  Graduate Texts in Mathematics, 108. Springer, 1986.

\bibitem{Schrijver1986}
Schrijver, A.  \textit{Theory of linear and integer programming.}  John Wiley \& Sons, 1986.

\bibitem{Sturmfels1993}
Sturmfels, B.  \textit{Algorithms in invariant theory.}  Texts and Monographs in Symbolic Computation.  Springer, 1993.

\end{thebibliography}
\end{document}